\documentclass{amsart}

\usepackage{amsmath}
\usepackage{amsfonts}
\usepackage{amssymb}
\usepackage{amsthm}

\usepackage{mathrsfs}
\usepackage{hyperref}

\newcommand{\<}[1]{\langle #1 \rangle}

\DeclareMathOperator{\Aut}{Aut}

\DeclareMathOperator{\GL}{GL}

\DeclareMathOperator{\M}{M}

\DeclareMathOperator{\Irr}{Irr}

\DeclareMathOperator{\Gal}{Gal}

\DeclareMathOperator{\GF}{GF}
\DeclareMathOperator{\lcm}{lcm}

\DeclareMathOperator{\pcb}{pcb}

\DeclareMathOperator{\pcbI}{pcbI}
\DeclareMathOperator{\PCBI}{PCBI}

\DeclareMathOperator{\PCB}{PCB}

\DeclareMathOperator{\Null}{Null}
\DeclareMathOperator{\Div}{Div}

\newtheorem{thm}{Theorem}[section]
\newtheorem{cor}[thm]{Corollary}

\newtheorem{lem}[thm]{Lemma}
\newtheorem{prop}[thm]{Proposition}

\newtheorem{defn}[thm]{Definition}
\newtheorem{notation}[thm]{Notation}

\newtheorem{remark}[thm]{Remark}

\title{Primary Cyclic Matrices in Irreducible Matrix Subalgebras}
\author[B.~Corr]{Brian P. Corr}
\address{Brian P. Corr, Centre for Mathematics of Symmetry and Computation,\newline
School of Mathematics and Statistics,\newline
The University of Western Australia,
 Crawley, WA 6009, Australia}
\email{brian.p.corr@gmail.com}

\author[C. E. Praeger]{Cheryl E. Praeger}
\address{Cheryl E. Praeger, Centre for Mathematics of Symmetry and Computation,\newline
School of Mathematics and Statistics,\newline
The University of Western Australia,
 Crawley, WA 6009, Australia\newline
Also affiliated with King Abdulaziz University,
Jeddah, Saudi Arabia} \email{Cheryl.Praeger@uwa.edu.au}

\thanks{The first author is supported by an Australian Postgraduate Award and UWA Top-Up Scholarship. This research forms part of Australian Research Council project DP110101153.}

\date{}
\begin{document}

\begin{abstract}
 {\it Primary Cyclic} matrices were used (but not named) by Holt and Rees in their version of Parker's \texttt{MEAT-AXE} algorithm to test irreducibility of finite matrix groups and algebras. They are matrices $X$ with at least one cyclic component in the primary decomposition of the underlying vector space as an $X$-module. Let $\M(c,q^b)$ be an irreducible subalgebra of $\M(n,q)$, where $n=bc >c$. We prove a generalisation of the Kung-Stong Cycle Index, and use it to obtain a lower bound for the proportion of primary cyclic matrices in $\M(c,q^b)$. This extends work of Glasby and the second author on the case $b=1$.\\\\
(2010 MSC Codes: 05A15, 15A30, 12F05, 20P05, 20C40)
\end{abstract}

\maketitle

\section{Introduction}
In order to improve and generalise the \texttt{MEAT-AXE} algorithm of Richard Parker \cite{parker}, Holt and Rees \cite{holtrees} suggested the use of a family of matrices defined as follows. An $n\times n$ matrix $X$ over a field $F=\GF(q)$ is {\it primary cyclic} if, for some irreducible polynomial $f$ over $F$, the nullspace of $f(X)$ in $V(n,q)=F^n$ is an irreducible $FX$-submodule (see also Definition \ref{PrimaryCyclic}).\\\\
Given a group $G\leqslant \GL(n,F)$ acting on $V=F^n$, the irreducibility test in the \texttt{MEAT-AXE} algorithm, originally due to Simon Norton, tests whether or not $G$ leaves invariant a proper nontrivial subspace of $V$. The version of the test used by Holt and Rees in \cite{holtrees} does so by randomly searching for primary cyclic matrices and analysing their action on $V$: for the analysis, then, it is crucial to know how abundant primary cyclic matrices are.\\\\
Holt and Rees in \cite[pp.7-8]{holtrees} obtain a positive constant lower bound on the proportion of primary cyclic matrices in the full matrix algebra $\M(n,F)$, and in \cite{GlasbyPraegerfcyclic} Glasby and the second author show that the proportion of primary cyclic matrices in $\M(n,F)$ lies in the interval $(1-\frac{c_1}{q^n}, 1-\frac{c_2}{q^n})$ for positive constants $c_1,c_2$. Here we focus on irreducible proper subalgebras of $\M(n,F)$: any such subalgebra can be identified with the full matrix algebra $\M(c,K)$ over some extension field $K=\GF(q^b)$, where $n=bc$ (see Section \ref{Prelim}). We prove an analogous result to the Holt-Rees estimate for these subalgebras.\\\\
We treat the case of fixed degree extensions $\GF(q^b)$ of a field of fixed size $q$ as the dimension $n=bc$ grows unboundedly. Let $P_M(c,q^b)$ be the proportion of matrices in $\M(c,q^b)$ which are primary cyclic in $\M(n,q)$ relative to some irreducible polynomial $f$ of degree $b$ over $F$ (note that this is the minimal possible degree of such an $f$): then $P_M(c,q^b)$ is a lower bound for the proportion of primary cyclic matrices in $\M(c,q^b)$.
\begin{thm}\label{maintheorem}
Let $q$ be a prime power, and $b,c$ positive integers with $b >1$. Then 
\begin{enumerate}
\item $\lim_{c\to\infty}P_M(c,q^b)$ exists and equals
\[P_M(\infty,q^b) :=\lim_{c\to\infty}P_M(c,q^b) = 1-\left(1-\frac{bq^{-b}}{(1-q^{-b})^2}\omega(1,q^b)^b\right)^{N(q,b)},\] where $\omega(1,q^b)=\prod_{i=1}^{\infty}(1-q^{-bi})$ and $N(q,b)$ is the number of monic irreducible polynomials of degree $b$ over $F_q$; and
\item there exists a constant $k(q,b)$ such that, if $c \geq  \left(\frac{\max\{b-1, q^b/b\}}{ \log (3/4)}\right)^2$, then 
\[\left|P_M(c,q^b)- P_M(\infty,q^b)\right| < k(q,b) q^{-bc}.\]
\end{enumerate}
\end{thm}
\begin{remark}\quad
\begin{enumerate}
\item To prove Theorem \ref{maintheorem}, we use generating functions and in particular, we obtain a new generalisation in Theorem \ref{ICycleIndexTheorem} of the Kung-Stong Cycle Index (see \cite{kung},\cite{stong}). 
\item Theorem \ref{maintheorem} shows that, for fixed $q,b$, the quantity $P_M(c,q^b)$ approaches its limiting value exponentially quickly. However the expression for the limit is rather complicated. We study the behaviour of the limiting value as $q^b$ grows, and prove (in Proposition \ref{LimitProp}) that the limit as $q^b$ approaches infinity of $P_M(\infty,q^b)$ exists and equals
\[ \lim_{q^b\to\infty} P_M(\infty,q^b) = 1-e^{-1}. \] 
This is analogous to the original Holt-Rees estimate in \cite{holtrees} for the case $b=1$.
\item We prove Theorem \ref{maintheorem}(ii) with the following value for the quantity $k(q,b)$: 
\[k(q,b) =\frac{8}{3(1-q^{-b})}\left(\frac{bq^b}{q^b-1}2^{2b}q^{2b^2}\right)^{q^b/b}\] 
(see Proposition \ref{Final}). We believe that this may be far from the best value.
\end{enumerate}
\end{remark}
Section 2 presents essential results on minimal and characteristic polynomials. Section 3 provides a generalisation of the Cycle Index Theorem and applies it to counting primary cyclic matrices in $\M(c,q^b)$. Section 4 deals with asymptotics and proves the second part of Theorem \ref{maintheorem}.\\\\
A consequence of Theorem \ref{maintheorem} is that, for sufficiently large $c$, an explicit lower bound on the proportion of primary cyclic matrices can be calculated. Computationally we determine the proportion exactly for small $n$, see for example, Table \ref{smalln}: combining these two methods we may address all values of $n$, so long as the field size $q^b$ is bounded.
\section{Preliminaries}\label{Prelim}
We first introduce some notation. Let $F$ be a field of order $q$ and let $K$ be an extension field of $F$ of degree $b$. The {\it Galois group} $G=\Gal(K/F)\leqslant \Aut K$ is cyclic of order $b$, generated by the Frobenius automorphism $\sigma_0:x\mapsto x^q$, and has the subfield $F$ as its fixed point set.\\\\
Let $V=F^n$ denote the space of $n$-dimensional row vectors over $F$, with standard basis $\{e_{1},\ldots,e_n\}$, and let $\M(n,q)$ denote the full endomorphism ring of $V$, with elements written as $n\times n$ matrices with entries in $F$ relative to the standard basis. For a divisor $b$ of $n$ (say $n=bc$), we can embed the algebra $\M(c,q^b)$ as an irreducible subalgebra of $\M(n,q)$ as follows. The extension field $K$ is an $F$-vector space of dimension $b$, having as a basis $\{1,\omega, \omega^2,\ldots,\omega^{b-1}\}$, where $\omega $ is a primitive element of $K$. If $\{v_1,\ldots, v_n\}$ is a basis for $V(c,q^b)=K^c$, then $\{ \omega^iv_j\mid 0\leq i \leq b-1, 1\leq j\leq c\}$ is an $F$-basis for $V(c,q^b)$ as an $n$-dimensional $F$-vector space, where $n=bc$, and the mapping $\varphi: \omega^iv_j \mapsto e_{(j-1)b+i+1}$ extends linearly to an $F$-vector space isomorphism from $V(c,q^b)=K^c$ to $V$.\\\\
Each $X\in\M(c,q^b)$ defines an $F$-endomorphism of $V(c,K)$, and so we have an action of $\M(c,q^b)$ on $V=F^n$ defined by 
\begin{equation}\label{Action}(v)X^{\varphi}:= v\varphi^{-1}X\varphi,\end{equation}
for $v\in V$. Thus $X\mapsto X^{\varphi}$ defines an $F$-algebra monomorphism $\M(c,q^b)\to\M(n,q)$, and we may identify $\M(c,K)$ with its image. This image is an {\it irreducible} $F$-subalgebra of $\M(n,q)$, and each irreducible subalgebra arises in this way (by Schur's Lemma, see for example \cite{dummit1999abstract}).
Throughout we will have to consider interchangeably the actions of a matrix in $\M(c,q^b)$ on {\it two} vector spaces, $F^n$ and $K^c$. For this reason we introduce notation to help keep track of which field we are dealing with. 
\begin{notation}\label{Notation}\quad
\begin{enumerate}
\item Let $V$ be the vector space $K^c$ of $c$-dimensional row vectors over $K=\GF(q^b)$, with $n=bc$. Then, as an $F$-vector space, $V$ is isomorphic, via $\varphi$ as defined above, to the vector space $F^n$. We denote this $F$-vector space by $V_F$. If there is any ambiguity we use $V_K$ to denote the $K$-vector space $V$. An element $X$ of $M(c,q^b)$ thus acts as a linear transformation of $V_F$ in a natural way (via the maps above): again we use the notation $X_F$ to denote the action of $X$ on $V_F$ (and similarly $X_K$ to denote the action on $V_K$ if there may be ambiguity).
\item We denote by $F[t], \Irr(q)$ and $\Irr(q,d)$ (where $d\geq 0$) the ring of polynomials over $F$, the set of monic irreducible polynomials over $F$, and the set of monic irreducibles of degree $d$ over $F$ respectively. Let $N(q,d)=|\Irr(q,d)|$. Denote the characteristic and minimal polynomials of $X_F$ by $c_{X,F}(t), m_{X,F}(t)$ respectively, and similarly define $K[t], \Irr(q^b,d), N(q^b,d)$ and $c_{X,K}(t), m_{X,K}(t)$ for the $X$-action on $V_K=K^c$.
\item The Galois group $G=\Gal(K/F)$ acts faithfully on $K[t]$ and $\M(c,q^b)$ by acting on the coefficients of a polynomial and the entries of a matrix respectively. The fixed points of $G$ in these actions are respectively $F[t]$ and $\M(c,q)$.
\item If $U$ is an $X$-invariant $F$-subspace of $V$, then we denote by $X|_U$ the restriction of $X$ to $U$; if in addition $U$ is a $K$-subspace then we may write $(X|_U)_F$ and $(X|_U)_K$ if we wish to emphasise the field.
\end{enumerate}
\end{notation}
\begin{defn}\label{primarydecomp}
Let $X\in\M(n,q)$ and let $m_{X,F}=\prod_{i=1}^r f_i^{\alpha_i}$, with each $f_i\in \Irr(q)$, and $\alpha_{i}>0$. A useful $X$-invariant decomposition of $V_F$ is the {\it $X$-primary decomposition} (see \cite[Theorem 11.8]{HartleyHawkes}):
\[V_F = V_{f_1} \oplus \dots \oplus V_{f_r}, \]
where the subspace $V_{f_i}$ is called the \emph{$f_i$-primary component} of $X$ (on $V$), and has the property that $f_i$ does not divide the minimal polynomial of the restriction of $X$ to $\oplus_{j\neq i} V_{f_j}$, and the minimal polynomial of $X|_{V_{f_j}}$ is $f_i^{\alpha_i}$. If an irreducible $f$ does not divide $c_{X,F}(t)$ we say the $f$-primary component is trivial and define $V_f=\{0\}$.
\end{defn}
We also define the $X_K$-primary decomposition of $V_K$ similarly.
\begin{defn}\label{PrimaryCyclic}
A matrix $X\in \M(n,q)$ is called {\it cyclic} if $m_{X,F}=c_{X,F}$, and, for $f\in \Irr(q)$, $X$ is {\it $f$-primary cyclic} if $X|_{V_f}$ is nontrivial and cyclic. Also, $X$ is {\it primary cyclic} if it is $f$-primary cyclic for some $f\in\Irr(q)$. We note that $X$ is $f$-primary cyclic if and only if the nullspace $\Null f(X)$ is an irreducible $FX$-submodule of $V$.
\end{defn}
\subsection{Minimal and Characteristic Polynomials}
We aim to count matrices $X$ in the subalgebra $\M(c,q^b)$ of $\M(n, q)$ such that $X_F$ is primary cyclic. To do so we derive necessary and sufficient conditions for this property which are intrinsic to their action on $K^c$: that is to say, conditions on $X_K$. Our analysis follows that of \cite[Section 5]{neumannpraegercyclic}.
We investigate the relationship between the characteristic and minimal polynomials of a matrix $X$ over the two different fields $F$ and $K$. We call two polynomials $g,g'$ in $K[t]$ {\it conjugate} if there exists $\sigma\in G=\Gal(K/F)$ such that $g^{\sigma}=g'$. Recall Notation \ref{Notation}.
\begin{lem}\label{split}
Let $f\in \Irr(q,d)$, let $b \geq 2$, and let $G=\<{\sigma_0} = \Gal(K/F)$. Suppose that $g\in\Irr(q^b)$ is a divisor of $f$ in $K[t]$. Then the following hold:
\begin{enumerate}
\item $\deg g = {d}/{\gcd(b,d)}$; 
\item $f=\lcm \{g^{\sigma_0^{i-1}}\mid 1\leq i\leq b \} = \prod_{i=1}^{\gcd(b,d)} g^{\sigma_0^{i-1}}$;
\item $g = g^{\sigma_0^i}$ if and only if $i\equiv 0 \pmod{\gcd(b,d)}$;
\item $f$ is the unique element of $\Irr(q)$ divisible by $g$ in $K[t]$.
\end{enumerate}
\end{lem}
\begin{proof}
Part (i) follows immediately from \cite[Theorem 3.46]{LidlNiederreiter}. For (ii) and (iii), observe that since $\sigma_0$ fixes the field $F$, the image $g^{\sigma_0}$ divides $f^{\sigma_0}=f$, and similarly, for every $i$ we have $g^{\sigma_0^i} \mid f$, so 
\[\lcm \{g^{\sigma_0^{i-1}}\mid 1\leq i\leq b \} \text{ divides } f.\]
Since the set $\{g^{\sigma_0^{i-1}}\mid 1\leq i\leq b \}$ is permuted under the action of $\sigma_0$, its least common multiple is fixed by $\sigma_0$, and so lies in $F[t]$. Then by the irreducibility of $f$, they are equal.\\\\
Since $\deg f=d=\gcd(b,d)\deg g$, it follows that $\{g^{\sigma_0^{i-1}}\mid 1\leq i \leq b\}$ has size $\gcd(b,d)$, and the stabiliser of each $g^{\sigma_0^{i-1}}$ in $G$ is $\<{\sigma_0^{\gcd(b,d)}}$. This implies part (iii) and the last assertion of (ii). Part (iv) follows from part (ii).
\end{proof}
The following is an immediate consequence of Lemma \ref{split}.
\begin{cor}\label{SplitCor}
Let $f,b,d,G,g$ be as in Lemma \ref{split}, and suppose that $b\mid d$. Then the following hold:
\begin{enumerate}
\item $\deg g = {d}/{b}$; 
\item $f=\lcm \{g^{\sigma_0^{i-1}}\mid 1\leq i\leq b \} = \prod_{i=1}^{b} g^{\sigma_0^{i-1}}$;
\item For every nontrivial $\sigma\in G$, $g\neq g^{\sigma}$.
\end{enumerate}
\end{cor}
We now give a description of $f$-primary cyclic matrices in terms of their representations over the field $K$. The following result uses ideas and information from the proof of \cite[Lemma 5.1]{neumannpraegercyclic}.
\begin{prop}\label{polynomials}
Let $f\in \Irr(q)$, let $G=\Gal(K/F)$, and let $X\in\M(c,q^b)$ such that $f$ divides $c_{X,F}(t)$. Then $X_F$ is $f$-primary cyclic if and only if $b\mid (\deg f)$ and the following hold for some divisor $g\in K[t]$ of $f$ of degree $(\deg f)/b$:
\begin{enumerate}
\item $X_K$ is $g$-primary cyclic; and
\item for every nontrivial $\sigma\in G$, we have that $g^{\sigma}\neq g$ and $g^{\sigma}$ does {\it not} divide $c_{X,K}(t)$.
\end{enumerate}
\end{prop}
\begin{proof}
Let $d=\deg f, r=\gcd(d,b)$, and $g\in \Irr(q^b)$ such that $g\mid f$. Let $g_i= g^{\sigma_0^{i-1}}$ for $1\leq i\leq r$. Then by Lemma \ref{split}, $f=\prod_{i=1}^r g_{i}$ and $\deg g=d/r$. Consider the $X_F$-invariant decomposition of $V$:
\[V_F = V_f \oplus V',\]
where $V_f$ is the $f$-primary component of $V$ and the minimal polynomial of $(X_F)|_{V_f}$ is $f^{\alpha}$. Comparing this to the $X_K$-invariant decomposition
\[V_K= V_1 \oplus V_2 = \left(\bigoplus_{i=1}^r V_{g_i}\right) \oplus V_2, \]
where for each $i$, the minimal polynomial of $X_K$ restricted to $V_{g_i}$ is $g_i^{\alpha_i}$ for some nonnegative integer $\alpha_i$, and the minimal polynomial of $X_K$ restricted to $V_2$ is not divisible by any $g_i$, we see that $(V_1)_F = V_f$, since by Lemma \ref{split}(iv), the $g_i$ are the only divisors of $f$. \\\\
By \cite[Lemma 5.1]{neumannpraegercyclic}, the minimal polynomial $m$ of $X_F$ restricted to $(V_1)_F$ is
\begin{equation}\label{PolynomialsEqn} \lcm\{ \prod_{i=1}^r (g_i^{\alpha_i})^{\sigma} \mid \sigma\in G\}.\end{equation}
By Lemma \ref{split}(ii), for each $i$, we have $\lcm\{g_i^{\sigma}\mid \sigma\in G\}= f$, and it follows that $m= f^{\max \alpha_i}$, where $\max\alpha_i = \max\{\alpha_i\mid 1\leq i\leq r\}$. Then since $V_f = (V_1)_F$, we have $\max\alpha_i = \alpha$.\\\\
Now suppose $X_F$ is $f$-primary cyclic. Recall from Definition \ref{PrimaryCyclic} that this is true if and only if the minimal polynomial of $(X|_{V_f})_F$ has degree equal to $\dim V_f$. Then $\alpha d = \dim V_f=\dim (V_1)_F$. Suppose that more than one of the $\alpha_i$ is positive. Then 
\[\frac{\alpha d}{b} = \frac{\dim (V_f)_F}{b} = \dim (V_1)_K\geq \deg\left( \prod_{i=1}^r g_i^{\alpha_i}\right)=  \sum_{i=1}\deg(g_i) \alpha_i ,\]
and so, using the fact that each $\deg(g_i) = d/\gcd(d,b)$ by Lemma \ref{split}, we have
\[\frac{\alpha d}{b}> \frac{d}{\gcd(d,b)} \max\alpha_i = \frac{ d\alpha}{\gcd(d,b)} \geq \frac{d \alpha}{b},\]
which is impossible. Hence only one of the $\alpha_i$ is nonzero, and so exactly one of the $g_i$ divides $c_{X,K}(t)$, say $g=g_1$, so $\alpha=\alpha_1$ and $V_f=(V_g)_F$. Then $\frac{\alpha d}{b}  =\frac{d}{\gcd(d,b)} \alpha$, implying that $b\mid d$. It follows from Corollary \ref{SplitCor} that $r=b,\prod_{i=1}^r g_i = \prod_{\sigma\in G} g_1^{\sigma}$, and $g^{\sigma}\neq g$ for all $\sigma\in G$. Thus (ii) holds.
Now $\dim (V_g)_F = \dim V_f$, and as we observed above, this equals $\alpha d$. Hence
\[\dim (V_g)_K = \alpha d /b = \deg g^{\alpha},\]
and so $X_K$ is $g$-primary cyclic so (i) also holds.\\\\
The converse is easier: if $b\mid d$ and $g$ is the only divisor of $f$ dividing $c_{X,K}(t)$, then $V_f = (V_g)_F$, and if also $X_K$ is $g$-primary cyclic and the minimal polynomial of $(X|_{V_g})_F$ is $g^{\alpha}$, then by (\ref{PolynomialsEqn}) the minimal polynomial $m$ of $(X|_{V_f})_F$ is $\lcm\{(g^{\sigma})^{\alpha}\mid \sigma\in G\}$. Since $g^{\sigma}\neq g$ for all nontrivial $\sigma\in G$, $m=\prod_{\sigma\in G} (g^{\sigma})^{\alpha}$, with degree $\alpha b\deg(g) = \alpha d$.\\\\
On the other hand, since $(V_g)_K$ is cyclic, it has dimension equal to the degree of the minimal polynomial $g^{\alpha}$, namely $\alpha\deg(g)=\frac{\alpha d}{b}$. Thus $\dim((V_g)_F)=\alpha d =\deg(m)$, so $V_f$ is cyclic.
\end{proof}
The next corollary follows immediately from Lemmas \ref{polynomials} and \ref{split}(iii).
\begin{cor}\label{polysets}
Let $X\in\M(c,q^b)\subseteq \M(n,q)$, where $n=bc$, let $G=\Gal(K/F)$, and let $I=\{f_1, \ldots,f_k\}\subset \Irr(q,b)$. Then $X_F$ is $f_i$-primary cyclic for every $i$ if and only if there exists a set $I'=\{g_1,\ldots,g_k\}\subseteq \Irr(q^b,1)$ with $|I'|=k$ satisfying the following for each $i\in \{1,\ldots , k\}$:
\begin{enumerate}
 \item $g_i|f_i$, and $X_K$ is $g_i$-primary cyclic;
 \item for every nontrivial $\sigma\in G$, we have $g_i^{\sigma}\neq g_i$, and $g_i^{\sigma}$ does not divide $m_{X,K}(t)$.
\end{enumerate}
\end{cor}
\section{A Generalised Cycle Index for Matrix Algebras}
Our main tool in enumerating matrices is the cycle index of the matrix
algebra $\M(n,q)$, introduced by Kung \cite{kung} and developed further by Stong \cite{stong}, and based on Polya's cycle index (see for example \cite{polya1987combinatorial}) of a permutation group. We continue to use Notation \ref{Notation}. To each pair $(h,\lambda)$, with $h\in\Irr(q)$ and $\lambda$ a partition of a nonnegative integer, denoted $|\lambda|$, with $|\lambda|\in [0,n]$,
assign an indeterminate $x_{h,\lambda}$. Then the {\it cycle index} of $\M(n,q)$ is
the multivariate polynomial \[Z_{\M(n,q)}(\mathbf{x}) := \frac{1}{|\GL(n,q)|}\displaystyle\sum_{X\in
\M(n,q)}\left(\displaystyle\prod_{h\in\Div X} x_{h,\lambda(X,h)}\right),\]
where $\mathbf{x}$ is a vector representing the set of indeterminates $x_{h,\lambda}$ occuring, $\Div X$ is the set of irreducible polynomials dividing $c_{X,K}(t)$ and $\lambda(X,h)$ is a partition (of an integer) uniquely determined by the structure of the action of $X$ on the primary component $V_h$ as described in Definition \ref{lambdadef} below.\\\\
In this section we generalise the Cycle Index of Kung and Stong to include variables associated with a finite number of irreducible polynomials which do not divide $c_{X,F}(t)$. We will apply this more general version in our study of primary cyclic matrices. We begin by presenting the original Cycle Index Theorem: we omit the proof, for it will follow immediately from our generalised version below. In this section $V=F^c$ is viewed solely as a $F$-space, where, recall, $F=\GF(q)$.
\begin{defn}\label{lambdadef}
Let $X\in\M(n,q), h\in\Irr(q)$, and let $\alpha_h$ be the multiplicity of $h$ in $c_{X,F}(t)$. Then $X$ acts on the $h$-primary component $V_h$ of $V_F$ with characteristic polynomial $h^{\alpha_h}$, and $\alpha_h \deg h = \dim (V_h)_F$ (so $\alpha_h=0$ if $V_h=0$). There is a direct sum decomposition of $V_h$ into $FX$-modules $V_h = V_{\lambda_1} \oplus \cdots \oplus V_{\lambda_r}$ with each $V_{\lambda_i}$ cyclic, such that the restriction of $X$ to $V_{\lambda_i}$ has minimal polynomial $h^{\lambda_i}$, and $\lambda_i \geq \lambda_{i+1}$ for all $i$. The $\lambda_i$ are uniquely determined by $X$ (see \cite[Theorem 11.19]{HartleyHawkes}). Define the partition $\lambda(X,h)$ as the ordered tuple
\[\lambda(X,h):= (\lambda_1, \lambda_2,\ldots ,\lambda_r, 0,0,\ldots ).\]
Then $\lambda(X,h)$ is a partition of $\dim V_h = \alpha_h \deg h$, and as this partition is non-increasing, we often omit the `trailing zeroes'  and write $(\lambda_1, \ldots, \lambda_r)$ if $V_h \neq \{0\}$ and $():=(0,0,\ldots)$ if $V_h=\{0\}$.\\\\ 
The partition $\lambda(X,h)$ is the empty partition (of the integer zero) if $h\not\in\Div X$, and otherwise is determined by the sizes of the blocks in the Frobenius Normal Form of $X|_{V_h}$.
\end{defn}
For more information on the cyclic and primary decompositions, and on $\lambda(X,h)$, see \cite{HartleyHawkes}. Lemma \ref{lambda} follows immediately from the definition of $\lambda(X,h)$:
\begin{lem}\label{lambda}
Let $X\in \M(n,q),h\in\Irr(q)$, and $\lambda=\lambda(X,h)$. Then the following hold:
\begin{enumerate}
\item $h\not\in \Div X$ if and only if $\lambda(X,h)=()$. In particular, $\deg h > n$ implies $\lambda(X,h)=()$; 
\item $h\in \Div X$ and $X$ is $h$-primary cyclic if and only if $\lambda(X,h)$ is $(\lambda_1)$, with $\lambda_1 > 0$, and in this case $(\deg h)\mid \lambda_1$ and $\lambda_1/\deg h$ is the multiplicity of $h$ in $c_{X,F}(t)$; and 
\item $h\in \Div X$ and $X$ is not $h$-primary cyclic if and only if $\lambda(X,h)$ has at least two nonzero parts.
\end{enumerate}
\end{lem}
\begin{defn}\label{clambdaq}
Let $\lambda$ be a partition of an integer $|\lambda|$, let $h\in \Irr(q)$, and let $s=|\lambda|\deg h$. If $\lambda=()$ then define $c(\lambda,\deg h,q)=1$. If $|\lambda|\geq 1$ then there exists a matrix $X:=X_{\lambda,h}\in \M(s,q^b)$ such that $c_{X,K}(t)=h^{|\lambda|}$, and the cyclic decomposition of $K^s$ described in Definition \ref{lambdadef} determines the partition $\lambda$. In this case we define
\[c(\lambda,\deg h,q) := |C_{\GL(s,q)}(X)|,\]
the number of matrices in $\GL(s,q)$ which commute with $X$. This quantity depends only on $\deg h$ and $\lambda$, since all such matrices are conjugate under elements of $\GL(s,q)$ (see again \cite[Theorem 11.19]{HartleyHawkes}). The number of such matrices $X$ is $\frac{|\GL(s,q)|}{c(\lambda, \deg h,q)},$ and this holds also for $\lambda=()$ if we take $\GL(0,q)$ as the trivial group.
\end{defn}
\begin{thm}[Cycle Index Theorem]
The generating function for the Cycle Index of a matrix algebra $\M(n,q)$ satisfies
\[1+\displaystyle\sum_{n=1}^{\infty} Z_{M(n,q)}(\mathbf{x})u^n = \displaystyle\prod_{h\in \Irr(q)} \left(1+\displaystyle\sum_{\lambda} x_{h,\lambda(h)}\frac{u^{|\lambda|\deg(h)}}{c(\lambda,\deg h,q)} \right),\]
where $c(\lambda,\deg h,q)$ is as in Definition \ref{clambdaq}, and the sum is over all partitions $\lambda\neq ()$.
\end{thm}
The Kung-Stong Cycle Index assigns, to every $X\in \M(n,q)$, the monomial
\[\prod_{h\in \Div X} x_{h, \lambda(X,h)}\]
and sums over $\M(n,q)$. Given a finite subset $I\subseteq \Irr(q)$, we generalise by forcing, for every $h\in I$, the term $x_{h,\lambda(X,h)}$ to appear in every expression assigned, whether or not $h$ divides $c_{X,F}(t)$. The reason for this generalisation will become apparent when we apply this to the proof of Lemma \ref{PCBI} in Section \ref{CountingSection}: it permits us to ask questions about whether some (fixed) $f\in \Irr(q)$ divides $c_X(t)$.
\begin{defn}\label{ICycleIndexDef}
Let $I\subseteq \Irr(q)$ be finite, and let $\lambda(X,h)$ be defined as in Definition \ref{lambdadef}. Then the {\it $I$-Cycle Index} of $\M(n,q)$ is defined as
\begin{equation}\label{ICycDef}Z_{\M(n,q)}^{(I)}(\mathbf{x}) := \frac{1}{|\GL(n,q)|}\displaystyle\sum_{X\in
\M(n,q)}\left(\displaystyle\prod_{h\in (\Div X) \cup I} x_{h,\lambda(X,h)}\right),\end{equation}
or equivalently
\begin{equation}\label{ICycleIndexDef2}Z_{\M(n,q)}^{(I)}(\mathbf{x}) := \frac{1}{|\GL(n,q)|}\displaystyle\sum_{X\in
\M(n,q)}\left(\left(\displaystyle\prod_{h\in\Div X} x_{h,\lambda(X,h)}\right)\left(\prod_{h\in I \setminus (\Div X)} x_{h,()}\right)\right).\end{equation}
\end{defn}
The Kung-Stong Cycle Index is precisely the $I$-Cycle Index with $I=\emptyset$. We now prove the $I$-Cycle Index Theorem.
\begin{thm}[The $I$-Cycle Index Theorem]\label{ICycleIndexTheorem}
For a finite subset $I\subseteq \Irr(q)$ and $\lambda(X,h)$ as in Definition \ref{lambdadef}, the generating function for the $I$-Cycle Index of $\M(n,q)$ satisfies
\begin{equation}\label{ICycleIndex}\begin{array}{rl}\displaystyle\prod_{h\in I}x_{h,()}+\displaystyle\sum_{n=1}^{\infty} Z_{M(n,q)}^{(I)}(\mathbf{x})u^n = &\displaystyle\prod_{h\in \Irr(q^b)\setminus I} \left(1+\displaystyle\sum_{\lambda\neq ()} x_{h,\lambda}\frac{u^{|\lambda|\deg(h)}}{c(\lambda,\deg h,q)} \right)  \\ &\quad\quad\times\displaystyle\prod_{h\in I} \left(x_{h,()}+\displaystyle\sum_{\lambda\neq ()} x_{h,\lambda}\frac{u^{|\lambda|\deg(h)}}{c(\lambda,\deg h,q)} \right),\end{array}\end{equation}
where the function $c(\lambda,\deg h,q)$ is as in Definition \ref{Action}, and the sums on the right hand side are over all partitions $\lambda\neq ()$.
\end{thm}
\begin{proof}
Our proof follows that of Stong in \cite{stong}. We consider the quantities in (\ref{ICycleIndex}) as power series in the variables $x_{h,\lambda}$, and treat $u$ as a constant. Note that since $I$ is finite, and for $X\in\M(n,q)$ the set $\Div X$ is finite, each $Z_{\M(n,q)}^{(I)}(\mathbf{x})$ on the left hand side of (\ref{ICycleIndex}), when expressed as in (\ref{ICycleIndexDef2}), is clearly a sum of products of finitely many of the $x_{h,\lambda}$. Recall that $c((),\deg h,q)=1$ for all $h\in\Irr(q)$, and so
\[x_{h,()} = x_{h,()} \frac{u^{0.\deg h}}{c((),\deg h,q)}.\]
Let $\{h_i\mid 1\leq i \leq t \}\subseteq\Irr(q)$, and let $\{\lambda_i \mid 1\leq i \leq t\}$ be a multiset of partitions such that $\lambda_i$ may be $()$ if $h_i\in I$, and otherwise $\lambda_i \neq ()$. For each $i$, let $n_i = |\lambda_i|\deg h_i$, and let $n=\sum_{i=1}^t n_i$.The coefficient of $\prod_{i=1}^t x_{h_i,\lambda_i}$ on the right hand side of (\ref{ICycleIndex}) is 
\begin{equation}\label{coefficient}\left(\prod_{i=1}^t \frac{1}{c(\lambda_i,\deg h_i,q)}\right)u^n. \end{equation}
On the other hand, the coefficient of $\prod_{i=1}^n x_{h_i,\lambda_i}$ on the left hand side of (\ref{ICycleIndex}) is equal to 1 if $n=0$, and otherwise is $\frac{u^n}{|\GL(n,q)|}$ times the number of matrices $X\in \M(n,q)$ having characteristic polynomial $\prod_{i=1}^t h_i^{|\lambda_i|}$, with $\lambda(X,h_i)=\lambda_i$ for each $i$. Each of these matrices $X$ is uniquely determined by the following data:
\begin{enumerate}
\item Its Primary Decomposition $V=V_{h_1}\oplus \cdots \oplus V_{h_n}$ has $\dim V_{h_i} = n_i$, noting that we may have $\lambda(X,h_i)=()$ if $h_i\in I$; and
\item for each primary component $V_{h_i}$, the partition $\lambda_i = \lambda(X_{h_i},h_i)$.
\end{enumerate}
There are exactly 
\[\frac{|\GL(n,q)|}{\prod_{i=1}^n |GL(n_i,q)|}\]
direct sum decompositions of $V$ with the appropriate dimensions, and on each part $V_{h_i}$, there are exactly $|\GL(n_i,q)|/c(\lambda_i,h_i,q)$ matrices $X_{h_i}$ with $\lambda(X_{h_i},h_i)=\lambda_i$, as noted in Definition \ref{clambdaq}. Thus the coefficient of $\prod_{i=1}^t x_{h_i,\lambda_i}$ on the left hand side of (\ref{ICycleIndex}) is
\[\frac{u^n}{|\GL(n,q)|}\cdot\frac{|\GL(n,q)|}{\prod_{{1\leq i\leq t}} |GL(n_i,q)|} \cdot \prod_{1\leq i\leq t }\frac{|\GL(n_i,q)|}{c(\lambda_i,h_i,q)}=\prod_{1\leq i\leq t} \frac{1}{c(\lambda_i,h_i,q)}u^n,\]
which equals (\ref{coefficient}).
\end{proof}
\section{Counting}\label{CountingSection}
By evaluating (\ref{ICycleIndex}) in Theorem \ref{ICycleIndexTheorem} at different values of $\mathbf{x}$, we can enumerate subsets of $\M(c,q^b)$ having certain properties based on their minimal polynomials. In particular, we wish to count matrices in $\M(c,q^b) \subseteq \M(n,q)$ which are $f$-primary cyclic for some $f\in\Irr(q,b)$ (recall that by Proposition \ref{polynomials}, $b$ is the smallest degree for which such an $f$ exists). We begin this section by introducing some quantities which will simplify our rather complicated calculations.\\\\
Note that while the $I$-Cycle Index Theorem was presented for the full matrix algebra $\M(n,q)$, it may be applied directly to the irreducible subalgebra $\M(c,q^b)$, provided that we treat $\M(c,q^b)$ in its own right, rather than as a subalgebra of $\M(bc,q)$.
\begin{defn}\label{defs}
Define the following quantities:
\[\begin{array}{rlr}
\omega_n(u,q) &:= \displaystyle\prod_{i=1}^{n} (1-uq^{-i})& \text{for }\{u\in\mathbb{C} : |u|<q\};\\
\omega(u,q)&:=\displaystyle\prod_{i=1}^{\infty} (1-uq^{-i})& \text{for }\{u\in\mathbb{C} :|u|<q\};\\
G(u,q,n)&:= 1+\displaystyle\sum_{\lambda\neq ()} \frac{u^{|\lambda|}}{c(\lambda,n,q)}&\text{for }\{u\in\mathbb{C} :|u|<1\};\\
P(u,q)&:=1+\displaystyle\sum_{n=1}^{\infty}\frac{u^n}{\omega_n(1,q)}&\text{for }\{u\in\mathbb{C} :|u|<1\};\\
S(u,q)&:=\displaystyle\sum_{n=1}^{\infty} \frac{u^{n}}{q^{n}(1-q^{-1})} & \text{ for }\{u\in\mathbb{C} :|u|<q\};
\end{array}\]
where $c(\lambda,n,q)$ is as defined in Definition \ref{clambdaq} and the sum for $G(u,q,n)$ runs over all partitions $\lambda\neq ()$. Note that $\omega_n(1,q)=\displaystyle\frac{|\GL(n,q)|}{|\M(n,q)|}$, and $\omega(1,q)=\displaystyle\lim_{n\to\infty}\frac{|\GL(n,q)|}{|\M(n,q)|}$ exists.
\end{defn}
These definitions function to simplify our rather complicated calculations later. The following results will be used to help with manipulation of the generating functions:
\begin{lem}\label{lemma3}
The following relations hold between the quantities in Definition \ref{defs}, for $|u|<1$, and in case (iii) for $|u|< q$:
\[\begin{array}{cl}
 \text{(i)} & G(u,q,1)= P(uq^{-1},q); \\\\
 \text{(ii)} & \prod_{h\in \Irr(q)} G(u^{\deg h},q,{\deg h}) = P(u,q); \\\\
\text{(iii)} & P(u,q) = \text{$\frac{1}{1-u} P(uq^{-1},q) = \prod_{i=0}^{\infty}(1-uq^{-i})^{-1}$};\\\\
\text{(iv)} & S(u,q^b) = \displaystyle{\frac{1}{(q^b-1)}\frac{u}{(1-uq^{-b})}};
%
%
%
%
\end{array}\]
\end{lem}
\begin{proof}
For (i), in (\ref{ICycleIndex}) set $I=\emptyset$, and for all $\lambda$, set $x_{h,\lambda}=0$ if $h\neq t-1$ and $x_{t-1,\lambda}=1$. Using (\ref{ICycDef}) we see that the right hand side of (\ref{ICycleIndex}) is equal to
$G(u,q,1)$, while the left hand side is
\[
1+\sum_{n=1}^\infty u^n\cdot \left(\frac{\mbox{\#\ unipotent elements in $\M(n,q)$}}{|\GL(n,q)|}\right)
\]
which by Steinberg's Theorem \cite[Theorem 6.6.1]{carter1993finite} is equal to $1+\sum_{n=1}^\infty 
\frac{u^nq^{n(n-1)}}{|\GL(n,q)|}$ and this equals $P(uq^{-1},q)$.\\
For (ii),  The left hand side of the equation in (ii) is equal to the right hand side of
(\ref{ICycleIndex}) if we set $I=\emptyset$ and all the  $x_{h,\lambda}=1$. Thus by (\ref{ICycleIndex}), using also (\ref{ICycDef}) and Definition \ref{defs}, 
this is equal to $1+\sum_{n=1}^\infty \frac{|\M(n,q)|}{|\GL(n,q)|} u^n = P(u,q)$.\\
(iii) In \cite[p.19]{andrews1998theory} we find the equality, for $|u|<q$,
\[
\prod_{r=1}^\infty ({1-uq^{-r}})^{-1}=1+\sum_{n=1}^\infty \frac{u^nq^{n(n-1)/2}}{\prod_{i=1}^n (q^i-1)}
\]
the right hand side of which is equal to $P(uq^{-1},q)$. This proves the second equality of (iii),
and the first equality follows on substituting $u$ for $uq^{-1}$ into the second equality.\\
Part (iv) is a routine geometric series calculation.
\end{proof}
\begin{defn}\label{quantities}\quad
\begin{enumerate}
\item For nonempty $I\subseteq\Irr(q,b)$, define \[\pcbI(I, c,q^b):= \{X\in\M(c,q^b)\mid \text{$X_F$ is $f$-primary cyclic for all $f\in I$\}};\]
\item Define $\pcb(c,q^b) := \displaystyle\bigcup_{\substack{I\subseteq \Irr(q,b)\\ I\neq \emptyset}} \pcbI(I,c,q^b)$;
\item Define probabilistic generating functions for $\pcbI$ and $\pcb$:
\[
\begin{array}{rl}
\PCBI(I,u,q^b) &:= 1+ \displaystyle\sum_{c=1}^{\infty} \frac{|\pcbI(I,c,q^b)|}{|\GL(c,q^b)|}u^c \\
\PCB(u,q^b) &:= 1+ \displaystyle\sum_{c=1}^{\infty} \frac{|\pcb(c,q^b)|}{|\GL(c,q^b)|}u^c. 
\end{array}
\]
 \end{enumerate}
\end{defn}
Note that $\pcb(c,q^b)$ is the set of matrices $X\in \M(c,q^b)$ such that $X_F$ is $f$-primary cyclic for some $f\in\Irr(q,b)$: hence the name `\textbf{p}rimary \textbf{c}yclic, degree $\mathbf{b}$'. Our end goal is to find and investigate $\PCB(u,q^b)$: to do so we compute a formula for $\PCBI(I,u,q^b)$, depending only on the size of $I$ and the parameters $q,b$, and a relationship between the functions $\PCB, \PCBI$.
\subsubsection{Finding the Generating Function $\PCBI(I,u,q^b)$}
\begin{lem}\label{PCBI}
Let $I=\{f_1,\ldots,f_k\}\subseteq\Irr(q,b)$, with $|I|=k$, and let $\PCBI(I,u,q^b)$ be as defined in Definition \ref{quantities}. Then for $|u|<1$, we have
\[\PCBI(I,u,q^b) = P(u,q^b)H(u,q^b)^k,\] where $H(u,q^b):= bP(u,q^b)^{-b}(1-u)^{-b}S(u,q^b)$, with $P(u,q^b),S(u,q^b)$ as in Definition \ref{defs}.
\end{lem}
\begin{proof}
Let $G=\Gal(K/F)$. By Corollary \ref{polysets}, a matrix $X_F$ is $f_i$-primary cyclic for all $i$ if and only if there exist divisors $g_i$ of $f_i$ for each $i\leq k$ such that $I'=\{g_1,\ldots,g_k\}\subseteq \Irr(q^b,1)$ has size $|I'|=k$, for each $i$, the $g_i$-primary component of $X_K$ is cyclic, and for $1\neq \sigma\in G$, $g_i^{\sigma}$ does not divide $m_{X,K}$. Fix a subset $I'$ and set
\[x_{h,\lambda} = 
\begin{cases} 
 0 & \text{ if $h\in I'$, and either $\lambda=()$ or $\lambda\neq (|\lambda|, 0,\ldots )$, with $|\lambda|>0$;}\\
 0 & \text{ if for some nontrivial $\sigma\in G$, $h^{\sigma}\in I'$;}\\
1 & \text{ if $h\in I', \lambda = (|\lambda|, 0,\ldots )$ with $|\lambda|>0$; and}\\
1 & \text{ if $h\not\in \cup_{\sigma\in G}(I')^{\sigma}$.}
\end{cases}\]
Let $X\in\M(c,q^b)$: then $X$ contributes 1 to the $I'$-Cycle Index (\ref{ICycDef}), evaluated at $\mathbf{x}$, if and only if, for every $g_i \in I'$, $\lambda(X,g_i) = (|\lambda|,0,\ldots)$, with $|\lambda | > 0$, and $\lambda(X,g_i^{\sigma})=()$ for all nontrivial $\sigma\in G$; and $X$ contributes zero otherwise. This is precisely the set of matrices which, for every $g_i\in I'$ and nontrivial $\sigma$, are $g_i$-primary cyclic and $g_i^{\sigma}\nmid m_{X,K}(t)$.\\\\
Arguing as in the proof of Theorem \ref{ICycleIndexTheorem} (and in particular noting (\ref{coefficient})), the number of matrices $X$ which contribute $1$ to the $I'$-cycle index is the same for each choice of the $k$-element set $I'$. By Corollary \ref{polysets}, each member of $\pcbI(I,c,q^b)$ contributes $1$ for a unique choice of $I'$. Since there are $b^k$ possible $I'$ corresponding to $I$, the number of $X\in \M(c,q^b)$ for which (\ref{ICycDef}) evaluates to $1$ with the above assignment of the $x_{h,\lambda}$ is therefore $|\pcbI(I,c,q^b)|/b^k$. Set $I^*= \cup_{\sigma\in G} (I')^{\sigma}$. Then since by Corollary \ref{polysets} we have $g^{\sigma} \neq g$ for every nontrivial $\sigma\in G$, we have $|I^*|=bk$. Hence, by Theorem \ref{ICycleIndexTheorem}, we have
\[\begin{array}{rl}\PCBI(u,q^b) &= b^k \displaystyle\prod_{h\in (\Irr(q^b)\setminus I^*)} \left(1+\displaystyle\sum_{\lambda\neq ()} \frac{u^{|\lambda|\deg(h)}}{c(\lambda,\deg h,q^b)} \right)\\ &\qquad\qquad\times\displaystyle\prod_{h\in I'} \left(\displaystyle\sum_{\lambda = (|\lambda|,0,\ldots)\neq ()} \frac{u^{|\lambda|\deg(h)}}{c(\lambda,\deg h,q^b)} \right).\end{array}\]
Now since every polynomial in $I'$ is linear, and by \cite[Table 1]{GlasbyPraegerfcyclic} we have that $c((|\lambda|,0,\ldots), 1, q^b) = q^{|\lambda| b}(1-q^{-b})$, it follows that
\[\begin{array}{rl} \displaystyle\prod_{h\in I'} \left(\displaystyle\sum_{\lambda = (|\lambda|,0,\ldots)} \frac{u^{|\lambda|\deg h}}{c(\lambda,\deg h,q^b)} \right) &= \displaystyle\prod_{h\in I'} \left(\displaystyle\sum_{\alpha=1}^{\infty} \frac{u^{\alpha}}{q^{\alpha b}(1-q^{-b})} \right)\\
&\\
&= S(u,q^b)^k. \end{array}\]
Then by Definition \ref{defs} and Lemma \ref{lemma3}, and since $|I^*|=bk$,
\[\begin{array}{rl}\PCBI(u,q^b) &= b^k S(u,q^b)^k \left(\displaystyle\prod_{h\in (\Irr(q^b)\setminus I^*)} G(u^{\deg h},q^b,\deg h)\right)\\
&= b^kS(u,q^b)^k\left(\displaystyle\prod_{h\in \Irr(q^b)} G(u^{\deg h},q^b,\deg h)\right)\left(\displaystyle\prod_{h\in I^*} G(u,q^b,1)\right)^{-1}\\
&= b^kS(u,q^b)^kP(u,q^b) P(uq^{-b}, q^b)^{-bk}\\
&= b^kS(u,q^b)^kP(u,q^b) ((1-u)P(u, q^b))^{-bk}\\
&= P(u,q^b)\left(bS(u,q^b)(1-u)^{-b}P(u, q^b)^{-b}\right)^k
\end{array}\]
and the result follows.
\end{proof}
\section{Combining Results}\label{combine}
The function $\PCBI(I,u,q^b)$ counts the number of elements of $\M(c,q^b)$ which are $f$-primary cyclic for (at least) $|I|$ distinct irreducibles of degree $b$ in $I$ (as elements of the larger algebra $\M(n,q)$, where $n=bc$). We seek the proportion of matrices which are $f$-primary cyclic for \emph{some} $f\in \Irr(q,b)$. The Inclusion-Exclusion Principle yields the following:
\begin{thm}\label{PCB}
For any $q,b$, let $H(u,q^b)= bP(u,q^b)^{-b}(1-u)^{-b}S(u,q^b)$, where $S(u,q^b),P(u,q^b)$ are as defined in Definition \ref{defs}. Then we have
\[\PCB(u,q^b) = P(u,q^b) \left(1-(1-H(u,q^b)^N\right),\]
where $N=|\Irr(q,b)|$.
\end{thm}
\begin{proof}
Any $X\in\M(c,q^b)$ which is primary cyclic as an element of $\M(n,q)$ relative to some element of $\Irr(q,b)$ lies in $\pcbI(I,c,q^b)$ for at least one nonempty subset $I$ of $\Irr(q,b)$. Thus for every $c$,
\[\pcb(c,q^b) = \bigcup_{I\subseteq\Irr(q^b)} \pcbI(I,c,q^b),\]
and by the inclusion-exclusion principle,
\[|\pcb(c,q^b)| = \displaystyle\sum_{i=1}^N (-1)^{i+1} \left(\sum_{{I\subseteq \Irr(q,b),|I|=i}}|\pcbI(I,c,q^b)|\right),\]
where $N = |\Irr(q,b)|$. By Lemma \ref{PCBI}, the value of $|\pcbI(I,c,q^b)|$ depends only on $|I|$. Thus
\[\sum_{{I\subseteq \Irr(q,b),|I|=i}}|\pcbI(I,c,q^b)| = \binom{N}{i}|\pcbI(I_i,c,q^b)|,\]
for some fixed $i$-element subset $I_i$ of $\Irr(q,b)$. Hence 
\[|\pcb(c,q^b)| = \displaystyle\sum_{i=1}^N (-1)^{i+1} \binom{N}{i}|\pcbI(I_i,c,q^b)|.\]
Since this relationship is a `linear combination', the same holds for the generating functions:
\[
\PCB(u,q^b) = \displaystyle\sum_{i=1}^N (-1)^{i+1} \binom{N}{i}|\PCBI(I_i,u,q^b)|,
\]
and so by Lemma \ref{PCBI}, writing $P=P(u,q^b)$ and $H=H(u,q^b)$, we have
\[
\begin{array}{rl}
 \PCB(u,q^b)&= P \left(\sum_{i=1}^N (-1)^{i+1} \binom{N}{i} PH^i\right)\\
&= P\left(1-\displaystyle\sum_{i=0}^N (-1)^i \binom{N}{i} H^i\right)\\
&= P\left(1- (1-H)^N\right)
\end{array}
\]
as required.
\end{proof}
Theorem \ref{PCB} allows us to easily compute (using, for example,{\it Mathematica} \cite{mathematica}) the Taylor coefficients of $\PCB(u,q^b)$, and hence values of $\frac{|\pcb(c,q^b)|}{|\M(c,q^b)|}$ for small $c$. We summarise some small cases in Table \ref{smalln}.
\begin{table}[ht]\label{smalln}\begin{center}
\[\begin{tabular}{|c|r|}
\hline
$c$ & $P_M(c,q^b)$\qquad\qquad\qquad\qquad\qquad\qquad\qquad \\
\hline
1   & $1-qq^{-b}$ \\
\hline
2   & $\frac{1}{2}+\left(\frac{3}{2}-\frac{b}{2}\right) q^{-b}+\left(-\frac{b}{2}-q+\frac{b q}{2}-\frac{q^2}{2}\right)q^{-2 b}+\left(-1+\frac{b q}{2}-\frac{q^2}{2}\right)q^{-3 b}  +qq^{-4 b}$ \\
\hline
3   & $\begin{split}\frac{2}{3}+\left(\frac{1}{3}
-\frac{q}{2}\right) q^{-b}
+ \left(\frac{4}{3}-\frac{b}{2}-\frac{b^2}{6}+q-\frac{b q}{2}\right)q^{-2 b} \\
+ \left(-\frac{1}{3}-\frac{b^2}{3}-\frac{b q}{2}+\frac{b^2 q}{6}-q^2+\frac{b q^2}{2}-\frac{q^3}{6}\right)q^{-3 b}\\
 + \left(-1-\frac{b^2}{3}+\frac{q}{2}-b q+\frac{b^2 q}{3}-q^2+b q^2-\frac{q^3}{3}\right)q^{-4 b}\\
+\left(-1+\frac{b}{2}-\frac{b^2}{6}-\frac{b q}{2}+\frac{b^2 q}{3}+b q^2-\frac{q^3}{3}\right)q^{-5 b}\\
+ \left(-\frac{b q}{2}+\frac{b^2 q}{6}+q^2+\frac{b q^2}{2}-\frac{q^3}{6}\right)q^{-6 b}
+ \left(1+q^2\right)q^{-7 b}
-qq^{-8 b}\end{split}$\\
\hline
\end{tabular}\]\caption{The proportion of Primary Cyclic matrices for some $f$ of degree $b$ in $\M(c,q^b)$. Observe that as $q^b$ grows, the proportions rapidly approach positive constant values.}\end{center} \end{table}
The data suggests that the proportion has a nonzero constant term, so for every triple $(c,q,b)$ the proportion is nontrivial. We turn to complex analysis to determine what happens as $c\to\infty$. The following appears, for example, in \cite{fnp}, as Lemma 1.3.3:
\begin{lem}\label{asymptoticslemma}
Suppose $g(u)=\sum a_n u^n$ and $g(u)=f(u)/(1-u)$ for $|u|<1$. If $f(u)$ is analytic with a radius of convergence $R>1$, then $a_n \to f(1)$, and $|a_n - f(1)|=O(d^{-n})$ for any $d<R$.
\end{lem}
We apply this Lemma to $\PCB(u,q^b)$ to obtain one of our main results:
\begin{proof}[Proof of Theorem \ref{maintheorem}(i).]
By Lemma \ref{PCB}, we have, writing $N=|\Irr(q,b)|$,
\[\PCB(u,q^b) = P(u,q^b)(1-(1-H(u,q^b))^{N}).\] 
Set $L(u,q^b)=(1-u)\PCB(u,q^b)$. By Lemma \ref{lemma3}(iii) and Definition \ref{defs} we have $L(u,q^b)= \omega(1,q^b)^{-1}(1-(1-H(u,q^b))^N)$.
Now by Lemma \ref{lemma3}, writing $S=S(u,q^b)$ and $P=P(u,q^b)$ for brevity,
\begin{equation}\label{HEquation} H(u,q^b) = bP^{-b}(1-u)^{-b}S=\frac{b}{q^{b}-1}\frac{u}{1-uq^{-b}}\prod_{i=1}^{\infty}(1-uq^{-bi})^b
\end{equation}
and the infinite product is convergent for all $|u|<q^b$. In particular, $H(1,q^b)$ exists, and 
\begin{equation}\label{H1Equation} H(1,q^b) = \frac{bq^{-b}}{(1-q^{-b})^2}\omega(1,q^b)^b. 
\end{equation}
It follows that
\[L(1,q^b) = \omega(1,q^b)^{-1}(1-(1-H(1,q^b))^N).\]
By Lemma \ref{asymptoticslemma}, we have $\lim_{c\to\infty}\frac{|\pcb(c,q^b)|}{|\GL(c,q^b)|}=L(1,q^b)$, and so
\[\lim_{c\to\infty}\frac{|\pcb(c,q^b)|}{|\M(c,q^b)|}=\omega(1,q^b)\lim_{c\to\infty}\frac{|\pcb(c,q^b)|}{|\GL(c,q^b)|} = 1-(1-H(1,q^b))^N,\]
and the result is proved.
\end{proof}
The following Lemma is used in estimating the asymptotics of $P_M(\infty,q^b)$ as $q^b$ grows:
\begin{lem}\label{InequalitiesLem}
\begin{enumerate}
\item For any $x \in [0, 1/4)$, we have
\[ \prod_{i=1}^{\infty} ( 1 - x) > 1- x - x^2 > 1/2. \]
\item For any integer $b \geq 1$ and for $x \in [0,\frac{1}{2}]$, we have that
\[ 1-2bx \leq (1-x - x^2)^b. \]
\item For any $x>1$, we have $\frac{x}{\log x} > x^{1/2}$.
\item For any $x \in (0,\frac{1}{2})$, we have
\[ \frac{1}{1-x} < 1 + x+2x^2.\]
\end{enumerate}
\end{lem}
\begin{proof}
\begin{enumerate}
\item By the Pentagonal Number Theorem \cite{andrews1983euler}, we have 
\[\begin{array}{rl}
\prod_{i=1}^{\infty} ( 1 - x) 
&= \left(\sum_{k=-\infty}^\infty(-1)^k x^{k(3k-1)/2}\right) \\
&= 1 -x-x^2 + x^5 +x^7 - x^{12} - \cdots \\
&> 1 -x-x^2 + x^5 + x^7 - \sum_{j=12}^{\infty} x^j.
\end{array}\]
Now the geometric series gives $\sum_{j=12}^{\infty} x^j = \frac{x^{12}}{1-x}$, and this is clearly less than $x^5+x^7$, since $x < 1/4$, and so the difference $\prod_{i=1}^{\infty} ( 1 - x)  - (1-x-x^2)$ is positive. The second inequality follows immediately.
\item Fix $b\geq 1$, and let $f(x):= (1-x-x^2)^b - (1-2bx)$: we seek to prove that $f$ is nonnegative for $x\in [0,\frac{1}{2}]$. Now $f'(x) = b(1-x-x^2)^{b-1}(-1-2x) + 2b = b(2-(1+2x)(1-x-x^2)^{b-1})$. Since $x\in [0,\frac{1}{2}]$, we have $(1+2x) \leq 2$, and $0 < (1-x-x^2)^{b-1} \leq 1$, and so their product is at most $2$. Thus $f'(x) \geq 0$ for all $x\in[0,\frac{1}{2}]$, and so $f(x)$ is nondecreasing. Since $f(0)=0$, it follows that $f(x)$ is nonnegative.
\item Let $f(x) = \frac{x^{1/2}}{\log x}$. Then 
\[
f'(x) = \frac{\log x - 2}{2x^{1/2}(\log x)^2},
\]
which, for $x>1$, is zero if and only if $x=e^2$. Since $\lim_{x\to 1^+} f(x) = \infty$, and $f(e^2) = e/2$, and $f(e^4) = e^2/4$, $f$ is decreasing for $1< x < e^2$, and increasing for $x>e^2$. Thus $f(x) \geq e/2>1$ for all $x > 1$, and the result follows.  
\item Since $x< 1/2$, the result is equivalent to $1 < (1-x)(1+x+2x^2) = 1 +x^2 -2x^3$, which holds if and only if $0 < x^2(1-2x)$, and this last inequality holds for all $x\in (0,\frac{1}{2})$.
\end{enumerate}
\end{proof}
\begin{lem}\label{InequalitiesLem2}
Let $t \geq 1, 0 < \epsilon < 1$. Then for all $c$ such that $c > \max \{ 1, \left(\frac{t}{\log (1-\epsilon)}\right)^2\}$, we have that
\[ c^t \leq (1-\epsilon)^{-c}. \]
\end{lem}
\begin{proof}
The result holds if and only if
\[ t\log c \leq -c \log(1-\epsilon), \]
and so, since $\log c > 0$ and since $0 < 1-\epsilon < 1$ implies $\log(1-\epsilon) < 0$, this is true if and only if
\[ - \frac{t}{\log (1-\epsilon)} \leq \frac{c}{\log c}.\]
Since, by Lemma \ref{InequalitiesLem}(iv), $c/\log c > c^{1/2}$ for all $c>1$, if also $c^{1/2} \geq - t/\log(1-\epsilon)$ then this inequality holds.
\end{proof}
\begin{prop}\label{LimitProp}
Let $P_M(\infty,q^b)=\displaystyle\lim_{c\to\infty}\frac{|\pcb(c,q^b)|}{|\M(c,q^b)|}$, where $b\geq 2$. Then
\[- \frac{4b}{eq^{b/2}} < P_M(\infty,q^b)- (1-e^{-1}) < \frac{1+b}{eq^{b}}+\frac{2(1+b)^2}{eq^{2b}},\]
so that
\[|P_M(\infty,q^b)- (1-e^{-1})| < 4e^{-1} bq^{-b/2}.\]
\end{prop}
\begin{proof}
By Theorem \ref{maintheorem}(i), we have $P_M(\infty,q^b) = 1-(1-H(1,q^b))^N$, with $H(1,q^b)$ as in (\ref{H1Equation}) above. We consider the behaviour of $(1-H(1,q^b))^N$ as $q$ and $b$ grow. Since $\omega(1,q^b) = \prod_{i=1}^{\infty} ( 1-q^{-bi})$, and since $q^{-b} \leq 1/4$, by Lemma \ref{InequalitiesLem}(i), we have
\[1-q^{-b}-q^{-2b} < \omega(1,q^b)<1 -q^{-b}.\]
Applying Lemma \ref{InequalitiesLem}(ii) with $x=q^{-b}$ gives
\begin{equation}\label{BoundingOmegaB}
1-2bq^{-b} < \omega(1,q^b)^b <1 -q^{-b}. 
\end{equation}
Now as $N:=N(q,b) = \frac{1}{b}\sum_{d\mid b} \mu(d)q^{d/b}$, we have $\frac{1}{b}(q^b - 2q^{b/2}) \leq N(q,b)\leq\frac{q^b}{b}$. Thus
\[(1-H(1,q^b))^{\frac{1}{b}q^b} \leq(1-H(1,q^b))^N \leq (1-H(1,q^b))^{\frac{1}{b}(q^b-2q^{b/2})},\]
and so (with $H$ denoting $H(1,q^b)$ for simplicity):
\[{\frac{q^b}{b}}\log(1-H) \leq N\log(1-H) \leq {\frac{1}{b}(q^b-2q^{b/2})}\log(1-H).\]
Using the inequality $1-\frac{1}{x} \leq \log x \leq x-1$, which holds for all $x>0$, we have
\[{\frac{q^b}{b}}\frac{H}{H-1} \leq N\log(1-H) \leq -{\frac{1}{b}(q^b-2q^{b/2})}H.\]
Substituting for $H$ using (\ref{H1Equation}) and rearranging gives
\[\frac{-\omega(1,q^b)^b}{(1-q^{-b})^2-bq^{-b}\omega(1,q^b)^b} \leq N\log(1-H)\leq -{\frac{1}{b}(q^b-2q^{b/2})}\frac{bq^{-b}}{(1-q^{-b})^2}\omega(1,q^b)^b.\]
Using the right inequality of (\ref{BoundingOmegaB}) and observing a geometric series gives
\[
\begin{array}{rl}
\displaystyle\frac{-\omega(1,q^b)^b}{(1-q^{-b})^2-bq^{-b}\omega(1,q^b)^b}  & > \displaystyle\frac{-(1-q^{-b})}{(1-q^{-b})^2-bq^{-b}(1-q^{-b})} \\
&= \displaystyle\frac{-1}{1-q^{-b} - bq^{-b}}\\
&= \displaystyle\frac{-1}{1-(1+b)q^{-b}}\\
\end{array}
\]
and applying Lemma \ref{InequalitiesLem}(iii) with $x=(1+b)q^{-b}$ gives $N\log(1-H) > -1 - (1+b)q^{-b}-2(1+b)^2q^{-2b}$.\\\\
On the other hand, we have, using the left inequality in (\ref{BoundingOmegaB}), and since $q^{b} > 4$ implies that $\frac{1}{(1-q^{-b})^2} < \frac{1}{(3/4)^2} = 16/9 < 2$, that
\[
\begin{array}{rl}
-{\frac{1}{b}(q^b-2q^{b/2})}\frac{bq^{-b}}{(1-q^{-b})^2}\omega(1,q^b)^b &= -(1-2q^{-b/2})\frac{\omega(1,q^b)^b}{(1-q^{-b})^2}\\
&< \displaystyle\frac{-(1-2q^{-b/2})(1-2bq^{-b})}{(1-q^{-b})^2} \\
&= -1 + \displaystyle\frac{2q^{-b/2} +2(b-1)q^{-b} -4bq^{-3b/2} +q^{-2b}}{(1-q^{-b})^2}\\
&< -1 + 2(2q^{-b/2} +2(b-1)q^{-b} -4bq^{-3b/2} +q^{-2b}).
\end{array}
\]
Since $-4bq^{-3b/2}$ is negative, and $2q^{-b} > q^{-2b}$, this is less than $ -1 + 4q^{-b/2} +4bq^{-b} $. Thus we have proved that
\[-1 - (1+b)q^{-b}-2(1+b)^2q^{-2b} < N\log(1-H) < -1 + 4q^{-b/2} +4bq^{-b}, \]
and so exponentiating,
\[\exp\left(-1 - (1+b)q^{-b}-2(1+b)^2q^{-2b}\right) < (1-H)^N < \exp\left(-1 + 4q^{-b/2} +4bq^{-b}\right). \]
Now for $0\leq x\leq 1$ we have $e^x \leq 1+x+\frac{3}{4}x^2$ and $e^{-x} > 1-x$ (see for example \cite[Lemma 2.3]{GuestPraeger}). The first inequality implies that
\[
\begin{array}{rl}
(1-H)^N &< e^{-1}(1+ 4q^{-b/2} +4bq^{-b} + \frac{3}{4}(4q^{-b/2} +4bq^{-b})^2)\\
&= e^{-1} + 4e^{-1} q^{-b/2} + 4e^{-1}(b+3)q^{-b} + 24e^{-1}bq^{-3b/2} + 12e^{-1}b^2 q^{-2b}\\
&< e^{-1} + 4be^{-1} q^{-b/2},
\end{array}
\]
and the second inequality gives
\[\begin{array}{rl} (1-H)^N &> e^{-1}(1-(1+b)q^{-b}-2(1+b)^2q^{-2b})\\
&= e^{-1}-e^{-1}(1+b)q^{-b}-2e^{-1}(1+b)^2q^{-2b}.
\end{array}\]
Recalling that $P_M(\infty, q^b) = 1-(1-H)^N$, the first inequality in the statement is proved by subtracting these two values from 1. The second inequality follows immediately from the first.
\end{proof}
\subsection{Proof of Theorem \ref{maintheorem}(ii)}
Finally we apply the method of Wall (see \cite{fnp}) to $\M(c,q^b)$ to prove the second part of our main result, which gives a useful lower bound on $\frac{|\pcb(c,q^b)|}{|\M(c,q^b)|}$ for sufficiently large $c$. The inequality we require is proved in Proposition \ref{Final}, thus completing the proof of Theorem \ref{maintheorem}. We introduce the following notation, following Fulman in \cite{fulman1997probability}: for a function $X(u)$ of a complex variable, we denote by $[u^c]X$ the coefficient of $u^c$ in the Maclaurin Series of $X$.
\begin{lem}\label{CoefficientsLemma}
Let $X(u)$ be an analytic function of a complex variable, and let $t$ be a positive integer. Then 
\begin{enumerate}
\item for all $c\geq 1$, we have
\[[u^c]\left(\frac{X(u)}{1-u}\right) = \sum_{i=0}^c [u^i]X(u).\]
\item Suppose there exist constants $a_1,a_2$ such that $|[u^c]X(u)| \leq a_1a_2^{-c}$, for all $c \geq 0$. Then for all $c\geq 0$, we have
\[|[u^c](X(u)^t)| \leq a_1^t (c+1)^{t-1} a_2^{-c}.\]
\end{enumerate}
\end{lem}
\begin{proof}
\begin{enumerate}
\item Let $x_i:=[u^i]X(u)$. Then
\[\begin{array}{rl} \frac{X(u)}{1-u} &= (x_0 +x_1u +\cdots ) ( 1+ u + u^2 +\cdots ) \\
&= x_0 + (x_0+x_1)u + (x_0+x_1+x_2)u^2 + \cdots
\end{array}\]
and (i) follows.
\item We proceed by induction on $t$. The result holds for $t=1$ by assumption. Let $x_{ij}:= [u^j]X(u)^i$, and suppose that $t\geq 2$ and that part (ii) holds for $X(u)^{t-1}$. Then
\[\begin{array}{rl}X(u)^{t} &= X(u)^{t-1}X(u) \\
&= (x_{t-1,0} +x_{t-1,1}u + \cdots)(x_{10} + x_{11}u + \cdots)\\
&= \displaystyle\sum_{c=0}^{\infty} \sum_{i=0}^c (x_{t-1,i})(x_{1,c-i})u^c,
\end{array}\]
and so by induction
\[\begin{array}{rl}
|[u^c]X(u)^t| &= \left|\displaystyle\sum_{i=0}^c x_{t-1,i}x_{1,c-i}\right|\\
&\leq \displaystyle\sum_{i=0}^c  (a_1^{t-1} (i+1)^{t-2} a_2^{-i} ). (a_1 a_2^{-(c-i)})\\
&= a_1^t \displaystyle\sum_{i=0}^c  ((i+1)^{t-2}  a_2^{-c})\\
&\leq a_1^t (c+1)^{t-1} a_2^{-c},
\end{array}\]
since $\sum_{j=1}^{c+1} j^{t-2} \leq (c+1)^{t-1}$, and the result follows by induction.
\end{enumerate}
\end{proof}
\begin{lem}\label{J1}
Let $J(u,q^b)=(1-uq^b) \PCB(uq^b, q^b)$. Then for $c \geq 2$, we have
\[ [u^c]J(u,q^b) = \left(\frac{|\pcb(c,q^b)|}{|\M(c,q^b)|}- \frac{|\pcb(c-1,q^b)|}{|\M(c-1,q^b)|}\right) q^{bc}.\]
\end{lem}
\begin{proof}
By definition of $J(u,q^b)$ we have
\[\begin{array}{rl} 
J(u,q^b)&=(1-uq^b) \sum_{c=1}^{\infty} \frac{|\pcb(c,q^b)|}{|\M(c,q^b)|} (uq^b)^c \\
&=\frac{|\pcb(1,q^b)|}{|\M(1,q^b)|}uq^b+ \sum_{c=2}^{\infty} \left(\frac{|\pcb(c,q^b)|}{|\M(c,q^b)|}- \frac{\pcb(c-1,q^b)}{|\M(c-1,q^b)|}\right)q^{bc}u^c.
\end{array}\] 
\end{proof}
The remainder of this section is devoted to finding an upper bound on $[u^c]J(u,q^b)$, and using this to prove Theorem \ref{maintheorem}(ii).
\begin{lem}\label{L}
Define $L(u,q^b):=\prod_{i=1}^{\infty}(1-uq^{-bi})= (P(u,q^b)(1-u))^{-1}$, and suppose $b>1$. Then
\[L(u,q^b)=\frac{1}{1-u}\left( 1+\sum_{c=1}^{\infty}\frac{(-1)^c q^{bc} u^c}{\prod_{i=1}^c(q^{bi}-1)}\right)\]
and for all $c\geq 1$, we have
\[|[u^c]L(u,q^b)| \leq a_L q^{-bc}, \]
where $a_L=2q^{b}$.
\end{lem}
\begin{proof}
The first assertion follows from \cite[Corollary 2.2]{andrews1998theory}. For the second, observe that
\[\begin{array}{rl}[u^c]L=1+\displaystyle\sum_{k=1}^c \frac{(-1)^k q^{bk}}{\prod_{i=1}^k(q^{bi}-1)} &= 1+\displaystyle\sum_{k=1}^c \left(\frac{(-1)^k (q^{bk}-1)}{\prod_{i=1}^k(q^{bi}-1)}+\frac{(-1)^k}{\prod_{i=1}^k(q^{bi}-1)}\right)\\
&=1+\displaystyle\sum_{k=1}^c \left(\frac{(-1)^k}{\prod_{i=1}^{k-1}(q^{bi}-1)}+\frac{(-1)^k}{\prod_{i=1}^k(q^{bi}-1)}\right)\\
&= 1- 1 + \displaystyle\frac{(-1)^c}{ \prod_{i=1}^c(q^{bi}-1)}\\
&= \displaystyle\frac{(-1)^c q^{-bc(c-1)/2}}{\prod_{i=1}^c(1-q^{-bi})},\end{array}\] as all but the first and last terms of the alternating sum cancel. Now for all $c$, we have both $q^{-bc(c-1)} \leq q^{b}.q^{-bc}$, and $\prod_{i=1}^c (1-q^{-bi})  > \prod_{i=1}^{\infty} (1-q^{-bi}) > 1/2$ by Lemma \ref{InequalitiesLem}(i), and so
\[|[u^c]L| \leq 2q^b. q^{-bc}.\]
\end{proof}
\begin{lem}\label{J2}
Let $J(u,q^b)$ be as defined in Lemma \ref{J1}, and suppose that $b>1$. Let $M_{q^b}= \left(\frac{\max\{b-1, q^b/b\}}{ \log (3/4)}\right)^2$: then for $c\geq M_{q^b}$, and $a_J =     \frac{8}{3}\left(\frac{bq^b}{q^b-1}2^b(2q^b)^bq^{b^2}\right)^{\frac{q^b}{b}} $ we have
\[|[u^c]J(u,q^b)| < a_J ,\] and hence
\[\left|\frac{\pcb(c+1,q^b)}{|\M(c+1,q^b)|}- \frac{\pcb(c,q^b)}{|\M(c,q^b)|}\right| < a_J q^{-bc}.\]
\end{lem}
\begin{proof}
Using Theorem \ref{PCB}, the observation that $P(uq^b,q^b)=P(u,q^b)(1-uq^b)^{-1}$, the definition of $H(uq^b,q^b)$ from the right hand side of (\ref{HEquation}) and Lemma \ref{lemma3}(iii), we have (with $N=|\Irr(q,b)|$)
\begin{align*}
J(u,q^b) 
&= (1-uq^b)P(uq^b,q^b)(1-(1-H(uq^b,q^b))^N)\\
&= P(u,q^b)\left[ 1 - (1- \frac{bq^b}{q^b-1}\frac{u}{1-u} \prod_{i=1}^{\infty}(1-uq^{b-bi})^{b})^{N}\right]\\
&= P(u,q^b)\left[ 1 - (1- \frac{bq^b}{q^b-1}\frac{u}{1-u} \prod_{i=0}^{\infty}(1-uq^{-bi})^{b})^{N}\right]\\
&= P(u,q^b)\left[ 1 - (1- \frac{bq^b}{q^b-1}\frac{u}{1-u} P(u,q^b)^{-b})^{N}\right]\\
&= P(u,q^b)\left[ 1 - \left(1- \frac{bq^b}{q^b-1}u(1-u)^{b-1} L(u,q^b)^{b}\right)^{N}\right], \addtocounter{equation}{1}\tag{\theequation} \label{JEquation}
\end{align*}
since $L(u,q^b) = ((1-u)P(u,q^b))^{-1}$ by definition.
By Lemma \ref{L}, $|[u^c]L| \leq a_Lq^{-bc}$, where $a_L=2 q^b$, and hence by Lemma \ref{CoefficientsLemma}(ii), $|[u^c]L^b|$ is bounded above by $a_L^b(c+1)^{b-1} q^{-bc}$. Then
\[
\begin{array}{rl} 
|[u^c]\left((1-u)^{b-1}L^b\right)| \leq& \displaystyle\sum_{k=0}^b \binom{b}{k}a_L^b(c-k+1)^{b-1} q^{-b(c-k)}\\
<& \displaystyle\sum_{k=0}^b \binom{b}{k} a_L^b(c+1)^{b-1} q^{-b(c-b)}\\
=& a_L^b(c+1)^{b-1} q^{-b(c-b)} \left(\displaystyle\sum_{k=0}^b \binom{b}{k}\right) \\
=& 2^ba_L^bq^{b^2} (c+1)^{b-1} q^{-bc}.
\end{array}
\]
Multiplication by $u$ `shifts' the coefficients, so that $c$ is replaced with $c-1$: that is,
\[\left|[u^c]\left(u(1-u)^{b-1} L(u,q^b)^{b}\right)\right| < 2^ba_L^bq^{b^2+b} c^{b-1} q^{-bc}.\] 
It follows that
\[
\left|[u^c]\left(\frac{bq^b}{q^b-1}u(1-u)^{b-1} L(u,q^b)^{b}\right)\right| 
< \frac{bq^{2b}}{q^b-1}2^ba_L^bq^{b^2} c^{b-1} q^{-bc},
\] 
and since subtracting the function from $1$ has no effect on the absolute value of any coefficients when $c\geq 1$, we have (for $c > 1$) that
\[
\left|[u^c]\left(1- \frac{bq^b}{q^b-1}u(1-u)^{b-1} L(u,q^b)^{b}\right)\right| < \frac{bq^b}{q^b-1}2^ba_L^bq^{b^2} c^{b-1} q^{-bc},
\] 
and so by Lemma \ref{InequalitiesLem2} with $t=b-1,\epsilon = 1/4$, we have, for $c \geq \left(\frac{b-1}{\log (3/4)}\right)^2$ (and hence $c>  1$),
\[
\left|[u^c]\left(1- \frac{bq^b}{q^b-1}u(1-u)^{b-1} L(u,q^b)^{b}\right)\right| < \frac{bq^b}{q^b-1}2^ba_L^bq^{b^2} \left( \frac{3q^b}{4} \right)^{-c}.
\] 
Again applying Lemma \ref{CoefficientsLemma}(ii), with $t=N$, and since by \cite{LidlNiederreiter}, $N \leq q^b/b$,  we have
 \[
 \left|[u^c]\left(1- \frac{bq^b}{q^b-1}u(1-u)^{b-1} L(u,q^b)^{b}\right)^{N}\right| <
 \left(\frac{bq^b}{q^b-1}2^ba_L^bq^{b^2}\right)^{\frac{q^b}{b}} \left(c+1\right)^{\frac{q^b}{b}} \left(\frac{3q^b}{4}\right)^{-c}.
 \]
Then setting $a_J = \frac{8}{3}\left(\frac{bq^b}{q^b-1}2^ba_L^bq^{b^2}\right)^{\frac{q^b}{b}}$ and again applying Lemma \ref{InequalitiesLem2} (with $c+1$ in place of $c$ and $t=q^b/b$), we have, for $c > \left(\frac{q^b}{b \log (3/4)}\right)^2$, that $(c+1)^{q^b/b} < (1-1/4)^{-c-1} = \frac{4}{3} (3/4)^{-c}$, and so
\[
\left|[u^c]\left(1- \frac{bq^b}{q^b-1}u(1-u)^{b-1} L(u,q^b)^{b}\right)^{N}\right| <\frac{3a_J}{8}.\frac{4}{3}\left(\frac{9q^b}{16}\right)^{-c} = \frac{a_J}{2} \left(\frac{9q^b}{16}\right)^{-c}.
\]
Now by (\ref{JEquation}), we may attain an expression for $J(u,q^b)$ by multiplying the above equation by $P(u,q^b)$: doing so, and recalling that by definition $[u^c]P(u,q^b) = \omega(c,q^b)^{-1} = \prod_{j=1}^c(1-q^{-bj})$, gives
\[\begin{array}{rl}\left|[u^c]J(u,q^b)\right| &< \displaystyle\sum_{i=0}^{c} \prod_{j=i}^c (1-q^{-bj})\frac{a_J}{2}\left(\frac{9q^b}{16}\right)^{-i} \\
&< \displaystyle\frac{a_J}{2} \left(\sum_{i=0}^c\left(\frac{9q^b}{16}\right)^{-i} \right)\\
&< a_J , \end{array}\]
since $\sum_{i=0}^c \left(\frac{9q^b}{16}\right)^{-i} < 2$ when $q^b \geq 4$.\\\\
The second assertion follows directly from Lemma \ref{J1}.
\end{proof}
\begin{prop}\label{Final}
Suppose $b \geq 2$, and let $a_J, M_{q^b}$ be as defined in Lemma \ref{J2}. Then for $c > M_{q^b}$, we have
\[
\left|\frac{|\pcb(c,q^b)|}{|\M(c,q^b)|}- \lim_{n\to\infty}\frac{|\pcb(c',q^b)|}{|\M(c',q^b)|}\right| \leq \frac{a_J}{1-q^{-b}}  q^{-bc}.
\]
\end{prop}
\begin{proof}
By Lemma \ref{J2}, we have $\left|\frac{\pcb(c+1,q^b)}{|\M(c+1,q^b)|}- \frac{\pcb(c,q^b)}{|\M(c,q^b)|}\right| < a_J q^{-bc}$, and so for every $c'> c> M_{q^b}$ we have
\[ \begin{array}{rl} \left|\frac{|\pcb(c',q^b)|}{|\M(c',q^b)|}- \frac{|\pcb(c,q^b)|}{|\M(c,q^b)|}\right| &\leq \sum_{m=c}^{c'-1}\left|\frac{|\pcb(m+1,q^b)|}{|\M(m+1,q^b)|}- \frac{|\pcb(m,q^b)|}{|\M(m,q^b)|}\right| \\
&<\sum_{m=c}^{c'-1}  a_J q^{-bm} \\
&=q^{-bc}a_J \left(\sum_{m=0}^{c'-c-1}  q^{-bm} \right) \\
& < q^{-bc}a_J \left(\sum_{m=0}^{\infty} q^{-bm}\right) \\
&= q^{-bc} a_J \left(\frac{1}{1-q^{-b}}\right).
\end{array}\]
\end{proof}
\bibliographystyle{plain}
\bibliography{references}
\end{document}